\documentclass[reqno,a4paper]{amsart}
\usepackage[utf8]{inputenc}
\usepackage[active]{srcltx}
\usepackage{verbatim}

\usepackage{newsymbol}
\let\emptyset \undefined
\let\ge       \undefined
\let\le       \undefined
\newsymbol\le          1336  
\newsymbol\ge          133E  
\newsymbol\emptyset    203F
\newsymbol\notle       230A
\newsymbol\notge       230B

\usepackage{esint,graphicx,color,epsfig,latexsym,longtable,amsmath,amscd,amsxtra,amssymb,amsxtra,exscale,xy,mathrsfs,wrapfig,bbm}\xyoption{all}
\usepackage[colorlinks,linkcolor={blue},citecolor={blue},urlcolor={red}]{hyperref}

\usepackage[fleqn,tbtags]{mathtools}
\mathtoolsset{showonlyrefs,showmanualtags}

\theoremstyle{plain}
\newtheorem{theorem}{Theorem}[section]

\newtheorem{lemma}[theorem]{Lemma}
\newtheorem{proposition}[theorem]{Proposition}
\theoremstyle{remark}

\newtheorem{example}[theorem]{Example}

\newtheorem{remark}[theorem]{Remark}

 \numberwithin{equation}{section}

\newcommand{\SCP}[2]{ {{\rm(SCP)}_{(#1,#2)}}}

\newcommand {\dreihalb} {{{}^{3}\!\!/\!{}_{2}}}
\newcommand {\einhalb} {{{}^{1}\!\!/\!{}_{2}}}
\newcommand {\nhalb} {{{}^{n}\!\!/\!{}_{2}}}
\newcommand {\Mhalb} {{{}^{M}\!\!/\!{}_{2}}}

\newcommand {\pihalbe} {{{}^{\pi}\!\!/\!{}_{2}}}

\newcommand {\ZZ} {\mathbb Z}

\newcommand {\EE} {\mathbb E}

  \newcommand {\la}{\lambda}

  \newcommand {\ga}{\gamma}
  \newcommand {\om}{\omega}
  \newcommand {\Om}{\Omega}

  \renewcommand {\O}{\Omega}

  \newcommand {\al}{\alpha}
  \newcommand {\eps}{\varepsilon}

\let\oldmarginpar\marginpar
\renewcommand\marginpar[1]{\oldmarginpar[\raggedleft\footnotesize #1]%
{\raggedright\footnotesize #1}}

\newcommand{\suchthat}{: \;}

\def\N{{\mathbb N}}
\def\Z{{\mathbb Z}}

\def\R{{\mathbb R}}
\def\C{{\mathbb C}}

\newcommand{\E}{{\mathbb E}}
\renewcommand{\P}{{\mathbb P}}
\newcommand{\F}{{\mathscr F}}

\let\trueH\H
\renewcommand{\H}{{\mathscr H}}

\newcommand{\beq}{\begin{equation}}
\newcommand{\eeq}{\end{equation}}
\newcommand{\bal}{\begin{aligned}}
\newcommand{\eal}{\end{aligned}}
\newcommand{\ben}{\begin{enumerate}}
\newcommand{\een}{\end{enumerate}}
\newcommand{\bit}{\begin{itemize}}
\newcommand{\eit}{\end{itemize}}
\newcommand{\bth}{\begin{theorem}}
\renewcommand{\eth}{\end{theorem}}
\newcommand{\bpr}{\begin{proposition}}
\newcommand{\epr}{\end{proposition}}
\newcommand{\ble}{\begin{lemma}}
\newcommand{\ele}{\end{lemma}}
\newcommand{\bpf}{\begin{proof}}
\newcommand{\epf}{\end{proof}}
\newcommand{\bex}{\begin{example}}
\newcommand{\eex}{\end{example}}
\newcommand{\bre}{\begin{example}}
\newcommand{\ere}{\end{example}}

\renewcommand{\Im}{{\rm Im}\,}
\newcommand{\D}{{\mathsf D}}
\newcommand{\Nul}{{\mathsf N}}
\newcommand{\Ran}{{\mathsf R}}
\newcommand{\calL}{{\mathscr L}}
\newcommand{\n}{\Vert}
\newcommand{\one}{{{\bf 1}}}
\newcommand{\embed}{\hookrightarrow}
\newcommand{\s}{^*}
\newcommand{\lb}{\langle}
\newcommand{\rb}{\rangle}

\newcommand{\limn}{\lim_{n\to\infty}}

\begin{document}

\allowdisplaybreaks

\title[The stochastic Weiss conjecture]
{The stochastic Weiss conjecture for bounded analytic semigroups}

\author{Jamil Abreu}

\address{Instituto de Matem\'atica, Estat\'{\i}stica e Computa\c{c}\~ao Cient\'{\i}fica\\
Universidade Estadual de Campinas, Rua S\'ergio Buarque de Holanda, 651 \\ Campinas SP \\ Brazil}
\email{jamil@ime.unicamp.br}
\author{Bernhard Haak}

\address{Institut de Math\'ematiques de Bordeaux\\
Universit\'e Bordeaux 1\\
351, cours de la Lib\'eration\\
F-33405 Talence cedex\\
France}
\email{bernhard.haak@math.u-bordeaux1.fr}

\author{Jan van Neerven}

\address{Delft Institute of Applied Mathematics\\
Delft University of Technology\\P.O. Box 5031\\2600 GA Delft\\The Netherlands}
\email{J.M.A.M.vanNeerven@TUDelft.nl}

\thanks{The second named author is partially supported by the ANR project ANR-09-BLAN-0058-01.
The third named author is supported by VICI subsidy 639.033.604
of the Netherlands Organisation for Scientific Research (NWO)}

\keywords{Stochastic Weiss conjecture, stochastic evolution equations, invariant measures,
$H^\infty$--calculus, Rademacher interpolation}

\subjclass[2000]{Primary: 93B28, Secondary: 35R15, 46B09, 47B10, 47D06}

\date\today

\begin{abstract}
Suppose $-A$ admits a bounded $H^\infty$-calculus of angle 
less than $\pihalbe$ on a Banach space $E$ which has Pisier's
property $(\alpha)$, let $B$ be a bounded linear operator from a Hilbert space $H$ into the
extrapolation
space $E_{-1}$ of $E$ with respect to $A$, and let $W_H$ denote an $H$-cylindrical Brownian motion. 
Let $\ga(H,E)$ denote the space of all $\ga$-radonifying operators from $H$ to $E$. We prove that
the following assertions are equivalent:
\begin{enumerate}
 \item 
the stochastic Cauchy problem
$dU(t) = AU(t)\,dt + B\,dW_H(t)$
admits an invariant measure on $E$;
 \item $(-A)^{-\einhalb} B \in \ga(H,E)$;
 \item the Gaussian sum $\sum_{n\in\Z} \gamma_{n}  2^{\nhalb} R(2^n,A)B$ converges 
in $\ga(H,E)$ in probability.
\end{enumerate}
This solves the stochastic Weiss conjecture of \cite{HaakNeerven}.
\end{abstract}

\maketitle

\section{Introduction}
Let $A$ be the generator of a  strongly continuous bounded analytic semigroup $S= (S(t))_{t\ge 0}$
on a Banach space $E$, let $F$ be another Banach space,
 and let $C:\D(A)\to F$ be a bounded operator.
If there exists a constant $M\ge 0$ such that
$$ \int_0^\infty \n CS(t)x\n_F^2\,dt \le M^2 \n x\n_E^2, \quad \forall x\in \D(A),$$
an easy Laplace transform
argument shows that $$\sup_{\la >0}\, {\la}^\einhalb \n C R(\la,A)\n_{\calL(E)} \le M.$$
Here, as usual, $R(\la,A) = (\la-A)^{-1}$ denotes the resolvent of $A$ at $\la$.

The celebrated {\em Weiss conjecture} in linear systems theory is the assertion that
the converse also holds. It was solved affirmatively for normal operators
$A$ acting on a Hilbert space by Weiss \cite{Weiss:conjectures}, for generators 
of analytic Hilbert space contraction semigroups with $F = \C$ by Jacob and Partington \cite{JP}, 
and subsequently for operators
admitting a bounded $H^\infty$-calculus of angle $<\pihalbe$ acting on an $L^p$-space, $1<p<\infty$,
by Le~Merdy \cite{LeMerdy:weiss-conj, LeMerdy:square-functions}. Counterexamples to the general statement 
were found by Jacob, Partington and Pott \cite{JPP}, 
Zwart, Jacob, and Staffans \cite{ZwartJacobStaffans:weak-admissibility}, and 
Jacob and Zwart \cite{JacobZwart:counterexable-Weiss-conj}.

Whereas the Weiss conjecture is concerned with {\em observation operators}, 
in the context of stochastic
evolution equations it is natural to consider a `dual' version of the
conjecture in terms of {\em control operators}.
To be more precise, we consider the following situation.
Let $W_H = (W_H(t))_{t\in[0,T]}$ be a cylindrical Brownian motion in a Hilbert space $H$ and
let $B\in \calL(H,E_{-1})$ be a bounded linear operator. Here,
$E_{-1}$ denotes the extrapolation space of $E$ with respect to $A$ (see Subsection \ref{ss:extrapol}).
The {\em stochastic Weiss conjecture}, proposed recently in \cite{HaakNeerven}, is the assertion that,
under suitable assumptions on the linear operator $A$,
the existence of an invariant measure for the linear stochastic Cauchy problem
\[    \left\{
\begin{aligned}
 dU(t) & = AU(t)\,dt + B\,dW_H(t), \qquad t\in[0,T], \\
  U(0) & = 0,
\end{aligned}
\right.  \leqno{\SCP{A}{B}}
    \] 
is equivalent to an appropriate condition on the operator-valued function $\la \mapsto  {\la}^\einhalb R(\la,A)B$.
This conjecture is justified by the observation (cf. Proposition \ref{prop:invariant} below)
that an invariant measure exists if and only if $t\mapsto S(t)B$ defines an element of the space
$\ga(L^2(\R_+;H),E)$ (see Subsection \ref{ss:radonifying} for the definition of this space).

In the paper just cited, an affirmative solution was given in the case where $A$ and $B$ 
are simultaneously diagonalisable.
The aim of this article is to prove the stochastic Weiss conjecture for the class of
operators admitting a bounded $H^\infty$-calculus of angle $<\pihalbe$.
Denoting by $S(E)$ the class of all sectorial operators $-A$ on $E$ of angle $<\pihalbe$ that are injective and
have dense range, our main result reads as follows.

\begin{theorem}\label{thm:equivalence}
Let $E$ have property $(\alpha)$ and assume that $-A \in S(E)$ admits a bounded $H^\infty$--calculus of angle
$<\pihalbe$ on $E$. 
Let  $B: H \to E_{-1}$ be a bounded operator.
Then the following assertions are equivalent:

\ben
\item\label{item:main:existence-inv-measure} $\SCP{A}{B}$ admits an invariant measure on $E$;
\item\label{item:main:A-onehalf-B-in-gamma} $(-A)^{-\einhalb} B \in \ga(H, E)$;
\item\label{item:main:weiss-condition} $\la\mapsto \la^{\einhalb}R(\la,A)B $ defines an element in
$\gamma(L^2(\R_+,\frac{d\la}{\la};H), E)$;
\item\label{item:main:dyadic-condition} for all $\la>0$ we have $R(\la,A)B\in \gamma(H,E)$ and the Gaussian sum 
\[ \sum_{n\in\Z} \gamma_{n}  2^{\nhalb} R(2^n,A)B    \]
converges in $\ga(H,E)$ in probability (equivalently, in $L^p(\Om;\ga(H,E))$ for some (all) $1\le p<\infty$).
\een
\end{theorem}

Since $B$ maps into the extrapolation space $E_{-1}$, some care has to be taken in giving a 
rigorous interpretations of these assertions. The details will be explained below.

In the special case when $E$ is a Hilbert space and $H$ is a separable 
Hilbert space with orthonormal basis $(h_k)_{k\ge 1}$,
condition \ref{item:main:existence-inv-measure} is equivalent to 
\begin{equation}\label{eq:invariant-Hilbert} 
\sum_{k= 1}^\infty \int_0^\infty \n S(t)Bh_k\n^2\,dt < \infty,
\end{equation}
and condition \ref{item:main:dyadic-condition} reduces to 
\begin{equation}\label{eq:dyad} \sum_{k=1}^\infty \sum_{n\in\Z}  2^{n} \n R(2^n,A)Bh_k\n^2 < \infty.
\end{equation}
Compared to the Weiss conjecture, we see that a uniform boundedness condition on ${\la}^\einhalb R(\la,A)B$
gets replaced by a (dyadic) square summability condition along $(h_k)_{k\ge 1}$ in \eqref{eq:dyad}; 
this is consistent
with the square summability condition along $(h_k)_{k\ge 1}$ in \eqref{eq:invariant-Hilbert}.

All spaces are real. When we use spectral arguments, we turn to the complexifications
without further notice. 

\section{Preliminaries}\label{sec:prelim}
In this section we collect some notations and results that will be used in the proof of 
Theorem \ref{thm:equivalence}.

\subsection{Property $(\al)$} A {\em Rademacher sequence} is a sequence of independent random
variables taking the values $\pm 1$ with probability $\einhalb$.
Let
$(r_j')_{j=1}^\infty$ and $(r_k'')_{k=1}^\infty$
be Rademacher sequences on probability spaces $(\Om',\P')$ and 
$(\Om'',\P'')$, and let $(r_{jk})_{j,k=1}^\infty$
be a doubly indexed Rade\-macher sequence on a probability space $(\Om,\F,\P)$.
It is important to observe that the sequence $(r_j'r_k'')_{j,k=1}^\infty$ is not
a Rademacher sequence. By standard randomisation techniques one proves (see, e.g., 
\cite{NW-alpha}):
 
\begin{proposition}\label{prop:Pisier}
For a Banach space $E$ 
the following assertions are equivalent:
\ben
\item[\rm(1)]
there exists a constant $C\ge 0$ such that for all finite sequences
$(a_{jk})_{j,k=1}^n$ in $\R$ and $(x_{jk})_{j,k=1}^n$ in $E$ we have
$$ \E'\E'' \Big\n \sum_{j,k=1}^n a_{jk} r_j' r_k'' x_{jk}\Big\n^2
\le C^2 \big(\max_{1\le j,k\le n}|a_{jk}| \big)^2\E'\E'' \Big\n \sum_{j,k=1}^n  r_j' r_k'' x_{jk}\Big\n^2;
$$
\item[\rm(2)] there exists a constant $C\ge 0$ such that for all finite sequences 
$(x_{jk})_{j,k=1}^n$ in $E$ we have
$$ \frac1{C^2}\E \Big\n \sum_{j,k=1}^n  r_{jk} x_{jk}\Big\n^2
\le \E' \E'' \Big\n \sum_{j,k=1}^n r_j' r_k'' x_{jk}\Big\n^2
\le C^2 \E \Big\n \sum_{j,k=1}^n  r_{jk} x_{jk}\Big\n^2.
$$
\een
\end{proposition}

A Banach space $E$ is said to have {\em property $(\al)$} if it satisfies the above equivalent conditions.
Examples of spaces having this property are Hilbert spaces and the spaces $L^p(\mu)$ with $1\le p<\infty$.
Property $(\al)$ was introduced by Pisier \cite{Pisier:lust}, who proved that a Banach lattice 
has property $(\al)$ if and only if it has finite
cotype. In particular, the space $c_0$ fails property $(\al)$.

\subsection{$\ga$-Boundedness}
A family $\mathscr{T}\subseteq \calL(E,F)$ is called {\em $\gamma$-bounded}
if there exists a constant $C\ge 0$ such that for all finite sequences $(T_n)_{n=1}^N$ in $\mathscr{T}$ and 
$(x_n)_{n=1}^N$ in $E$ we have
$$ \E \Big\n \sum_{n=1}^N \ga_n T_n x_n \Big\n^2 \le C^2 \E \Big\n
\sum_{n=1}^N \ga_n x_n \Big\n^2.$$ 
The least admissible constant in this inequality is called the {\em $\ga$-bound} of $\mathscr{T}$.

By letting $N{=}1$ it is seen that $\gamma$-bounded families are uniformly bounded. For Hilbert spaces $E$ and $F$,
the notions of uniform boundedness and $\gamma$-boundedness are equivalent. 
For detailed expositions on
$\ga$-boundedness and the closely related notion of $R$-boundedness, as well as 
for references to the extensive literature we refer the reader
to \cite{ClementDePagterSukochevWitvliet, DHP, KunstmannWeis:Levico, Weis:FM}.

\subsection{$\ga$-Radonifying operators}\label{ss:radonifying}
Let $\H$ 
be a Hilbert space and $E$ a Banach space.
For a finite rank operator $T: \H \to E$ of the form
$$ T = \sum_{n=1}^N h_n \otimes x_n,$$
where $(h_n)_{n=1}^N$ is an orthonormal sequence in $\H$ and $(x_n)_{n=1}^N$ is a sequence in $E$,
we define
\begin{equation}\label{eq:finite-rank} \n T\n_{\ga(\H,E)} 
:= \Big\n\sum_{n=1}^N \ga_n x_n\Big\n_{L^2(\Om;E)}.\end{equation}
Here, $(\ga_n)_{n=1}^N$ is a sequence of independent standard Gaussian random variables 
on a probability space $(\Om,\P)$.
The Banach space $\ga(\H,E)$ is defined as the completion of the linear space of finite 
rank operators with respect to this norm.

The following {\em $\ga$-Fatou lemma} holds (see \cite{KaltonWeis:square-function-est, NeeCanberra}). 
Suppose $(T_n)_{n=1}^\infty$ is a bounded sequence in $\ga(\H,E)$
and $T\in\calL(\H,E)$ is an operator such that $$\limn \lb T_n h, x\s\rb = \lb Th,
x\s\rb,\quad \forall h\in \H, \ x\s\in E\s.$$ Then, if $E$ does not contain a closed subspace 
isomorphic to $c_0$, we have $T\in \ga(\H,E)$ and
\begin{equation}\label{eq:gFatou}\n T\n_{\ga(\H,E)}\le \liminf_{n\to\infty}\n T_n\n_{\ga(\H,E)}.\end{equation}

The {\em Kalton--Weis extension theorem} \cite[Proposition 4.4]{KaltonWeis:square-function-est} 
(see also \cite{NeeCanberra}) 
asserts that if $T: H_1\to H_2$ is a bounded linear operator, then the 
tensor extension $T: H_1\otimes E\to H_2\otimes E$, 
$$ T(h\otimes x) := Th\otimes x$$
extends to a bounded operator (with the same norm) from $\ga(H_1,E)$ to $\ga(H_2,E)$. 

The {\em Kalton--Weis multiplier theorem} \cite[Proposition 4.11]{KaltonWeis:square-function-est} 
(see \cite{NeeCanberra} for the formulation given here) 
asserts that if $(X,\mu)$ is a $\sigma$-finite measure space,
$E$ and $F$ are Banach spaces with $F$ not containing a closed subspace 
isomorphic to $c_0$, and if $M: X \to \calL(E,F)$
is measurable with respect to the strong operator topology and has $\ga$-bounded range, then 
the mapping 
$$(\one_B\otimes h)\otimes x \mapsto (\one_B\otimes h) \otimes Mx$$ has a unique 
extension to a bounded linear operator 
from $\ga(L^2(X,\mu;H),E)$ into $\ga(L^2(X,\mu;H),F)$ (with norm equal to the $\ga$-bound of the range of $M$).

Below we shall use (see \cite{NW-alpha}) that a Banach space $E$ has property $(\alpha)$ 
if and only if, whenever $\H_0$ and $\H_1$ are nonzero
Hilbert spaces, the mapping $(h_0\otimes h_1)\otimes x \mapsto h_0 \otimes (h_1\otimes x)$ 
extends to an isomorphism
of Banach spaces
$$ \ga(\H_0\widehat\otimes \H_1, E) \simeq \ga(\H_0, \ga(\H_1,E)).$$
Here,   $\H_0\widehat\otimes \H_1$ denotes the Hilbert space completion of the 
algebraic tensor product $\H_0\otimes\H_1$. We will be particularly interested in the case 
$\H_0=L^2(\R_+,\frac{dt}{t})$, in which case the above isomorphism then takes the form
$$\ga(L^2(\R_+,\tfrac{dt}{t};H),E)\simeq \ga(L^2(\R_+,\tfrac{dt}{t}),\ga(H,E)).$$

\subsection{Stochastic integration}

Let $H$ be a Hilbert space
and let $(\Om,\P)$ be a probability space. A {\em cylindrical Brownian motion in $H$}
is a mapping $W_H: L^2(\R_+;H)\to L^2(\O)$ such that $W_H f$ is a centred Gaussian random variable for all 
$f\in L^2(\R_+;H)$ and 
$$ \E (W_H f \cdot W_H g) = [f,g]_{L^2(\R_+;H)}$$
for all $f,g\in L^2(\R_+;H)$. Such a mapping is linear and bounded. 

A function $\Phi: \R_+\to \calL(H,E)$ is said to be {\em stochastically integrable} with respect to $W_H$
if it is scalarly square integrable, i.e., for all $x\s\in E\s$ the function $\Phi\s x\s:t\mapsto 
\Phi\s(t)x\s$ belongs to $L^2(\R_+;H)$, and 
for all Borel sets $B\subseteq \R_+$ there exists a random variable $X_B\in L^2(\Om; E)$ 
such that 
$$ \int_B  \Phi\s x\s\,dW_H := W_H(\one_B  \Phi\s x\s) = \lb X_B, x\s\rb, \quad \forall x\s\in E\s.$$
In that case we define $$\int_B \Phi\,dW_H := X_B.$$

The following result was proved in \cite{vanNeervenWeis:stoch-int}.
\begin{proposition}
A scalarly square integrable function  $\Phi: \R_+\to \calL(H,E)$ is stochastically 
integrable with respect to $W_H$
if and only if there exists an operator $R\in \ga(L^2(\R_+;H),E)$ such that 
$R\s x\s = \Phi\s x\s$ in $L^2(\R_+;H)$ for all $x\s\in E\s$.
\end{proposition}

\subsection{Existence, uniqueness and invariant measures}\label{ss:extrapol}
Let $A$ be the generator of a strongly continuous semigroup $S = (S(t))_{t\ge 0}$ on a 
Banach space $E$. We define $E_{-1}:= (E\times E)/\mathscr{G}(A)$, where 
$\mathscr{G}(A) = \{(x,Ax): \ x\in \D(A)\}$ is the graph of $A$. The mapping 
\[    i_{-1}: x\mapsto (0,x)+\mathscr{G}(A)    \]
defines a dense embedding $i_{-1}$ from $E$ into $E_{-1}$. 
We shall always identify $E$ with it image $i_{-1}(E)$ in $E_{-1}$.

The operator $A$ extends to a bounded operator $A_{-1}$ from
$E$ into $E_{-1}$ by defining \[    A_{-1} x:= (-x,0)+\mathscr{G}(A).    \] 
To see that this indeed gives an extension
of $A$, note that for $x\in \D(A)$ we have 
\[    i_{-1} Ax = (0,Ax)+\mathscr{G}(A) = (-x,0)+\mathscr{G}(A) = A_{-1}x.    \]
It is easy to see that the operator $A_{-1}$, which is densely defined and closed as a linear operator 
in $E_{-1}$ with domain $\D(A_{-1}) = E$, generates a strongly continuous semigroup 
$S_{-1} = (S_{-1}(t))_{t\ge 0}$ on $E_{-1}$ which satisfies $S_{-1}(t)i_{-1}x 
= i_{-1}S(t)x$ for all $x\in E$ and $t\ge 0$.

For a bounded operator $B: H\to E_{-1}$ we are interested in $E$-valued solutions 
to the stochastic evolution equation
$\SCP{A}{B}$. To formulate this problem rigorously, we first consider the problem
$\SCP{{A_{-1}}}{B}$ in $E_{-1}$:
\[    \left\{
\begin{aligned}
 dU_{-1}(t) & = A_{-1}U_{-1}(t)\,dt + B\,dW_H(t), \qquad t\in[0,T], \\
  U_{-1}(0) & = 0.
\end{aligned}
\right.  \leqno{{\SCP{A_{-1}}{B}}}
    \]
Here, as always, $W_H$ is a cylindrical Brownian motion in $H$, and we adopt the standard notation 
$W_H(t)h:=W_H(\one_{(0,t)}\otimes h)$.

An $E$-valued process $U = (U(t))_{t\in [0,T]}$ is called a 
{\em weak solution} of $\SCP{A}{B}$ if the $E_{-1}$-valued process $i_{-1} U = (i_{-1}U(t))_{t\in [0,T]}$
is a weak solution of  $\SCP{A_{-1}}{B}$, i.e., for all $x_{-1}^*\in \D(A_{-1}^*)$ 
the function $t\mapsto \lb  i_{-1} U(t), A_{-1}^* x_{-1}^*\rb$ is integrable almost surely and
if for each $t\in [0,T]$ we have, almost surely,
\[     \lb  i_{-1} U(t),x_{-1}^*\rb = \int_0^t \lb  i_{-1} U(s), A_{-1}^* x_{-1}^*\rb\,ds
+ W_H(t) B\s x_{-1}\s.    \]

An $E$-valued process $U$ is called a 
{\em mild solution} of $\SCP{A}{B}$ if the $E_{-1}$-valued process $i_{-1} U$
is a mild solution of  $\SCP{A_{-1}}{B}$, i.e., if the function $t\mapsto S_{-1}(t)B$ is stochastically
integrable in $E_{-1}$ with respect to $W_H$ and if 
for each $t\in [0,T]$ we have, almost surely,
\begin{equation}\label{eq:mild}  i_{-1} U(t) = \int_0^t S_{-1}(t-s)B\,dW_H(s). \end{equation}

The following proposition is an extension of the main result of \cite{vanNeervenWeis:stoch-int} 
(where the case $B\in \calL(H,E)$ was considered). 

\begin{proposition}\label{prop:existence} Under the above assumptions, 
for an $E$-valued process $U$ the following assertions are equivalent:
\begin{enumerate}
 \item $U$ is weak solution of $\SCP{A}{B}$;
 \item $U$ is mild solution of $\SCP{A}{B}$;
 \item there exists an operator $R_T\in \ga(L^2(0,T;H),E)$ such that for all 
$x_{-1}^* \in E_{-1}^*$
\begin{equation}\label{eq:RT}    
 R_T^*  (i_{-1}\s x_{-1}\s) =  B\s S_{-1}^*(\cdot)x_{-1}\s \ \hbox{ in } L^2(0,T;H). 
\end{equation}
\end{enumerate}
\end{proposition}

\begin{proof}
 Let us prove the equivalence (b)$\Leftrightarrow$(c), because this is what we need in the sequel.
The proof of (a)$\Leftrightarrow$(b) is left to the reader.

(b)$\Rightarrow$(c): By assumption there is a strongly measurable random variable $U(T):\Omega
\to E$ such that in $E_{-1}$ we have $$i_{-1} U(T) = \int_0^T S_{-1}(T-s)B\,dW_H(s).$$
For all $ x_{-1}\s\in E_{-1}\s$, the random variable $\lb U(T), i_{-1}\s x_{-1}\s\rb$ is Gaussian.
Since $F := \{  i_{-1}\s x_{-1}\s: \  x_{-1}\s\in E_{-1}\s\}$ is weak$\s$-dense 
in $E\s$ and the range of $U(T)$ is separable up to a null set, 
from \cite[Corollary 1.3]{BrzNee} it follows
that $\lb U(T),x\s\rb$ is Gaussian for all $x\s\in E\s$, i.e., $U(T)$ is Gaussian distributed. 

By the results of \cite{vanNeervenWeis:stoch-int} the operator $R_{-1,T}:L^2(0,T;H) \to E_{-1}$, defined by
$$ R_{-1,T} f = \int_0^T S_{-1}(T-s)B f(s)\,ds,$$ belongs to $\ga(L^2(0,T;H),E_{-1})$.
Define the linear operator $R_T\s : F\to L^2(0,T;H)$ by
$$ R_T\s  i_{-1}\s x_{-1}\s := R_{-1,T}\s x_{-1}\s.$$
Then,
\begin{equation}\label{eq:RTs}
\begin{aligned}
 \n R_T\s  i_{-1}\s x_{-1}\s \n_{L^2(0,T;H)}^2  
& = \n R_{-1,T}\s x_{-1}\s \n_{L^2(0,T;H)}^2
\\ & = \int_0^T \n B\s S_{-1}\s(T-s) x_{-1}\s\n_H^2\,ds 
\\ & = \E \Big| \int_0^T  B\s S_{-1}\s(T-s) x_{-1}\s \,dW_H(s)\Big|_H^2 
\\ & = \E \lb U(T), i_{-1}\s x_{-1}\s\rb^2 = \n i_{T}\s i_{-1}\s x_{-1}\s \n_{\H_T}^2,
\end{aligned}
\end{equation} 
where $i_{T}$ is the canonical inclusion mapping
of the reproducing kernel Hilbert space $\H_T$, 
associated with the Gaussian random variable $U(T)$, into $E$.
This shows that $R_T\s$ is well-defined and bounded on $F$.
At this point we would like to use a density argument to infer that $R_T^*$ extends to 
a bounded operator from $E\s$ into $L^2(0,T;H)$
which satisfies
\begin{equation}\label{eq:covar} 
\n R_T\s x\s \n_{L^2(0,T;H)}^2 = \n i_{T}\s x\s\n_{\H_T}^2, \quad \forall x\s\in E\s.
\end{equation}
However, this will not work, since $F$ is only weak$\s$-dense in $E\s$. The correct way to proceed is as follows.
The injectivity of $i_{-1}\circ i_T$ implies that $i_T\s \circ i_{-1}\s$ has weak$\s$-dense range
in $\H_T$. As $\H_T$ is reflexive, this range is weakly dense and therefore, by the Hahn-Banach theorem,
it is dense. Fixing an arbitrary $x\s\in E\s$, we may choose a sequence $(x_{-1,n}^*)_{n\ge 1}$ 
in $E_{-1}^*$ such that 
$i_T\s i_{-1}\s x_{-1,n}^* \to i_T\s x\s$ in $\H_T$. By \eqref{eq:RTs} the sequence
$(R_T\s i_{-1}^* x_{-1,n}^*)_{n\ge 1}$ is Cauchy in $L^2(0,T;H)$ and converges to some $f_{x\s}\in L^2(0,T;H)$.
It is routine to check that $f_{x\s}$ is independent of the approximating sequence. Thus we may
extend the $R_T\s$ to $E\s$ by putting 
$$R_T\s x\s:= f_{x\s}.$$ 
Clearly, for this extended operator the identity \eqref{eq:covar} is obtained. 

We claim that its adjoint $R_T^{**}: L^2(0,T;H)\to E^{**}$ actually takes values in $E$,
and that this operator is the one we are looking for. 

First, for $f = \one_{(a,b)}\otimes h$ and $x\s\in E\s$ of the form
$x\s = i_{-1}\s x_{-1}\s$ we have
\begin{align*}\lb x\s, R_T^{**} f\rb & = [ R_T\s i_{-1}\s x_{-1}\s, f]_{L^2(0,T;H)} 
\\ & = \int_a^b \lb S_{-1}(T-s) Bh, x_{-1}\s\rb\,ds  = \lb i_{-1}y, x_{-1}\s\rb = \lb y,x\s\rb,
\end{align*}
where $y =  \int_a^b S_{-1}(T-s)Bh\,ds$ belongs to $\D(A_{-1}) = E$.
It follows that $R_T^{**}$ maps the dense subspace  of all $H$-valued step 
functions into $E$,
and therefore it maps all of $L^2(0,T;H)$ into $E$.

Viewing $R_T: = R_T^{**}$ as an operator from $L^2(0,T;H)$ to $E$, we finally note that
the identity \eqref{eq:covar} 
exhibits $R_T \circ R_T\s  = i_T\circ i_T\s$ as the covariance operator of the $E$-valued Gaussian 
random variable
$U(T)$. This means that $R_T$ is $\ga$-radonifying as an operator from $L^2(0,T;H)$ to $E$ 
(see, e.g., \cite{NeeCanberra}). 

(c)$\Rightarrow$(b): We follow the ideas of \cite{vanNeervenWeis:stoch-int}.
We have $L^2(0,T;H) = \Nul(R_T) \oplus \overline{\Ran(R_T\s)}$.
By the general theory
of $\ga$-radonifying operators, $G:=\overline{\Ran(R_T\s)}$ is separable (see \cite{NeeCanberra}).
By a Gram-Schmidt argument we may select a sequence $(x_{-1,n}\s)_{n\ge 1}$ 
in $E_{-1}\s$ such that $(g_n)_{n\ge 1}:= (R_T\s i_{-1}^*x_{-1,n}\s)_{n\ge 1}$ 
is an orthonormal basis for $G$. Then the Gaussian random variables
$$ \gamma_n := \int_0^T B\s S_{-1}\s(T-s)x_{-1,n}^*\,dW_H(s)$$
are independent and normalised.
Since $R_T$ is $\ga$-radonifying, the $E$-valued random variable 
$$ U(T) := \sum_{n\ge 1} \ga_n R_T g_n$$ 
is well-defined, and it is easy to check that it satisfies \eqref{eq:mild} with $t$ replaced by $T$.
By well-known routine arguments, this is enough to assure that $\SCP{A}{B}$ has a mild solution $U$ in $E$.   
\end{proof} 

Suppose now that the problem $\SCP{A_{-1}}{B}$ admits a mild solution $U_{-1}$ in $E_{-1}$
and let $\mu_{-1,t}$ denote the distribution of the random variable $U_{-1}(t)$.
The weak limit $\mu_{-1,\infty}$ of these measures, if it exists, is called the (minimal) {\em invariant
measure} associated with $\SCP{A_{-1}}{B}$. Thus, by definition, the invariant measure, if it exists,
is the unique Radon probability measure on $E_{-1}$ which satisfies
$$ \int_{E_{-1}} f\,d\mu_{-1,\infty} 
= \lim_{t\to\infty} \int_{E_{-1}} f\,d\mu_{-1,t}, \quad \forall f\in C_{\rm b}(E_{-1}).$$  
For an explanation of this terminology and a more systematic approach we refer the reader to
\cite{DaPZab}. This references deals with Hilbert spaces $E$; extensions of the linear theory
to the Banach space setting were presented in \cite{GoldysNeerven, vanNeervenWeis:asymptotic-scp}. 

A Radon probability measure $\mu$ on $E$ is an {\em invariant measure} for $\SCP{A}{B}$ if
the image measure $i_{-1}(\mu)$ on $E_{-1}$ is an invariant measure for $\SCP{A_{-1}}{B}$.
Extending a result from \cite{vanNeervenWeis:asymptotic-scp} (where the case
$B\in \calL(H,E)$ was considered) we have the following result. A proof is obtained along the same
line of reasoning as in the previous proposition and is left as an exercise to the reader.

\begin{proposition}\label{prop:invariant} Under the above assumptions, 
for a Radon probability measure $\mu$ on $E$ the following assertions are equivalent:
\begin{enumerate}
 \item $\SCP{A}{B}$ admits an invariant measure;
 \item there exists an operator $R_\infty\in \ga(L^2(\R_+;H),E)$ such that for all 
$x_{-1}^* \in E_{-1}^*$
\begin{align}\label{eq:Rinfty}  
 R_\infty^*  (i_{-1}\s x_{-1}\s) =  B\s S_{-1}^*(\cdot)x_{-1}\s \ \hbox{ in } L^2(\R_+;H).    
\end{align} 
\end{enumerate}
\end{proposition}

Formally, \eqref{eq:RT} and \eqref{eq:Rinfty} express that the
operators $R_T$ and $R_\infty$ are integral operators with kernels
$S(\cdot)B$. Strictly speaking this makes no sense, since $B$ maps
into $E_{-1}$ rather than into $E$. It will be convenient, however, to
refer to $R_T$ and $R_\infty$ as the operators `associated with
$S(\cdot)B$' and we shall do so in the sequel without further warning.

\subsection{Sectorial operators and $H^\infty$-calculus}
For $\theta\in (0,\pi)$ let 
$$\Sigma_{\theta}:= \{z\in \C\setminus\{0\}:\ |\arg(z)| < \theta\}$$
denote the open sector of angle $\theta$. 
A densely defined closed linear operator $-A$ in a Banach space $E$ is called {\em sectorial 
(of angle $\theta\in (0,\pi)$)} if the spectrum of $-A$ is contained in 
$\overline{\Sigma_\theta}$
and 
$$ \sup_{z\not\in \overline{\Sigma_{\theta}}}  \n z \,(z+A)^{-1}\n < \infty.$$
The infimum of all $\theta\in (0,\pi)$ such that $-A$ is sectorial of angle $\theta$
is called the {\em angle of sectoriality} of $-A$.

It is well known (see \cite[Theorem II.4.6]{EngelNagel}) that 
$-A$ is sectorial of angle less than $\pihalbe$ 
if and only if $A$ generates a strongly continuous bounded analytic semigroup on $E$.  

Following \cite{KunstmannWeis} we denote by $S(E)$ the set of all
densely defined, closed, injective operators in $E$ that are sectorial
of angle less than $\pihalbe$ and have dense range.  The injectivity
and dense range conditions are not very restrictive: if $A$ is a
sectorial operator on a reflexive Banach space $E$, then we have the
direct sum decomposition
$$ E = \Nul(A) \oplus \overline{\Ran(A)}$$
in terms of the null space and closure of the range of $A$.
In that case, the part of $A$ in $\overline{\Ran(A)}$ is sectorial and satisfies the additional
injectivity and dense range conditions.

Let $-A\in S(E)$ be sectorial of angle $\theta\in (0,\pihalbe)$ and
fix $\eta\in (\theta,\pihalbe)$.  We denote by
$H_0^\infty(\Sigma_\eta)$ the linear space of all bounded analytic
functions $f:\Sigma_\eta\to \C$ with some power type decay at zero and
infinity, i.e., for which there exists an $\eps>0$ such that $$|f(z)|
\le C|z|^\eps / (1+|z|)^{2\eps}, \quad \forall z\in \Sigma_\eta.$$ For
such functions we may define a bounded operator
$$ f(-A) = \frac{1}{2\pi i}\int_{\partial \Sigma_{\eta'}} f(z) (z+A)^{-1}\,dz,$$
with $\eta'\in (\theta,\eta)$. The operator $-A$ is said to have a bounded $H^\infty$-calculus 
if there exists a constant $C$, independent of $f$, such that
$$ \n f(-A)\n \le C\n f\n_\infty, \quad \forall f\in H_0^\infty(\Sigma_\eta).$$
The infimum of all admissible $\eta$ is called the {\em angle} of the $H^\infty$-calculus of $-A$.

Examples of operators $A$ for which $-A$ has a bounded $H^\infty$-calculus of angle less than $\pihalbe$ 
are generators of strongly continuous analytic contraction semigroups on Hilbert spaces
and second order elliptic operators on $L^p$-spaces whose coefficients satisfy mild regularity assumptions.
We refer to \cite{DHP, Haase:Buch, KunstmannWeis:Levico} for more details and examples. 

If $-A \in S(E)$ has a bounded $H^\infty$-calculus, the mapping $f\mapsto f(-A)$ extends (uniquely, in some natural sense
discussed in \cite{KunstmannWeis:Levico}) 
to a bounded algebra homomorphism from $H^\infty(\Sigma_\eta)$ into $\calL(E)$ of norm at most $C$.
A proof the following result can be found in  \cite{KunstmannWeis:Levico}.

\begin{proposition}\label{prop:gamma-bounded}
Suppose that $-A \in S(E)$ admits a bounded $H^\infty$-calculus of angle $\eta<\pihalbe$ and let 
$\eta< \eta'<\pihalbe$. Then 
$-A$ is {\em $\gamma$-sectorial} of any angle $\eta<\eta'<\pihalbe$, i.e., 
the family 
$$\{z \,(z+A)^{-1}: \ z\not\in \overline{\Sigma_{\eta'}}\}$$ is $\gamma$-bounded. 
If, in addition, $E$ 
has property $(\al)$, then the family
$$ \{f(-A): \ f\in H^\infty(\Sigma_{\eta'}), \ \n f\n_\infty\le 1\}$$
is $\ga$-bounded. 
\end{proposition}

\subsection{Rademacher interpolation}
If $-A$ is a sectorial operator
on $E$, then for $\theta\in\R$ we may define
the Banach space $\dot E_{\theta}$ as the completion of $\D((-A)^\theta)$
with
respect to the norm \[     \n x\n_{\dot E_{\theta}} := \n (-A)^{\theta} x\n.    \]
Note that $(-A)^{\theta}$ extends uniquely to
an isomorphism from $\dot E_{\theta}$ onto $E$; with some abuse of notation this
extension will also be denoted by  $(-A)^{\theta}$.
In particular, $\dot E_{-1}$ is the completion of the range 
$\Ran(A)$ with respect
to the norm
\[    \n Ax\n_{\dot E_{-1}} := \n x\n.    \]

Note that
\begin{equation}\label{eq:dot}     E+\dot E_{-1}=i_{-1} E_{-1}    
\end{equation}
 with equivalent norms.
For the reader's convenience we include the short proof.
We trivially have $E\embed E_{-1}$,
and the embedding $\dot E_{-1}\embed E_{-1}$ is a consequence of the fact that
for all $x\in \D((-A)^{-1}) = 
\Ran(A)$, say $x = Ay$, we have
\[     \n x\n_{E_{-1}} \le C\n (I-A)^{-1}x\n = C\n (I-A)^{-1}A\n \n y\n
= C\n (I-A)^{-1}A\n \n x\n_{\dot E_{-1}}.    \]
It follows that $ E+\dot E_{-1}\embed E_{-1}$
with continuous inclusion. Since $I-A$ is surjective from $E$ onto $E_{-1}$,
every $x\in E_{-1}$ is of the form $x = y-Ay$ for some $y\in E$, which implies
that $x \in E+\dot E_{-1}$. It follows that  the inclusion
$ E+\dot E_{-1}\embed E_{-1}$ is surjective, and the claim now follows from the
open mapping theorem.

Let $(X_0,X_1)$ be an interpolation couple of Banach spaces.
Let $(r_n)_{n\in\Z}$ be a Rademacher sequence on a probability space
$(\Omega,\P)$.
For $0<\theta<1$ the {\em Rademacher interpolation space}
$\lb X_0,X_1\rb_\theta$
consists of all $x\in X_0+X_1$ which can be represented as a sum
\beq\label{eq:reprx}
x = \sum_{n\in\Z} x_n, \quad x_n\in X_0\cap X_1,
\eeq
convergent in $X_0+X_1$, such that
\[    
\bal
\mathscr{C}_0((x_n)_{n\in\Z}) & := \sup_{N\ge 0} \E\Big(\Big\n\sum_{n=-N}^N
r_n 2^{-{n\theta}} x_n\Big\n_{X_0}^2\Big)^\einhalb < \infty,\\
\mathscr{C}_1((x_n)_{n\in\Z}) & := \sup_{N\ge 0} \E \Big(\Big\n\sum_{n=-N}^N
r_n 2^{n(1-\theta)} x_n\Big\n_{X_1}^2\Big)^\einhalb < \infty.
\eal
    \]
The norm of an element  $x\in \lb X_0,X_1\rb_\theta$ is defined as
\[     \n x\n_{\lb X_0,X_1\rb_\theta} :=
\inf \  \Big(\max\big\{\mathscr{C}_0((x_n)_{n\in\Z}), \,
\mathscr{C}_1((x_n)_{n\in\Z})\big\}\Big),    \]
where the infimum extends  over all representations \eqref{eq:reprx}.
This interpolation method was introduced by Kalton, Kunstmann and Weis, who proved that
if $-A$ admits a bounded $H^\infty$--calculus (of any angle $<\pi$), 
then
for all $0<\theta<1$ and real numbers $\al<\beta$ one has
\[     \lb \dot E_\al, \dot E_{\beta}\rb_\theta = \dot E_{(1-\theta)\al + \theta \beta}    \]
with equivalent norms
\cite[Theorem 7.4]{KaltonKunstmannWeis}.
Applying this 
to the induced operator $I \otimes A$ on $L^2(\Omega;E)$,
defined by $(I\otimes A)(f\otimes x):= f\otimes Ax$ for $f\in L^2(\Omega)$ and vectors $x\in \mathscr{D}(A)$,
we obtain the following vector-valued extension of this result:

\begin{proposition}\label{prop:RadInterp}
If $-A \in S(E)$ admits a bounded $H^\infty$--calculus, then
\[      \lb L^2(\Omega; \dot E_\al), L^2(\Omega;\dot E_{\beta})\rb_\theta = 
L^2(\Omega;\dot E_{(1-\theta)\al + \theta \beta}).    \]
\end{proposition}

\section{Proof of Theorem~\ref{thm:equivalence}}

We begin with a useful observation.

\begin{lemma}\label{lem:resolvB}
 Let $A$ generate a strongly continuous semigroup on $E$ and suppose that 
the equivalent conditions of Proposition~\ref{prop:invariant} be satisfied.
Then  for all $\lambda\in \varrho(A)$ there exists an operator $\widehat S(\la)B\in \ga(H,E)$
such that  \[     i_{-1}\circ \widehat S(\la)B =  R(\lambda,A_{-1})\circ B.    \]
\end{lemma}
\begin{proof}
It suffices to prove this for one $\la\in\varrho(A)$; then, 
by the resolvent identity, this holds for all $\la\in\varrho(A)$.

% Indeed, for $\mu\in\varrho(A)$ we then define
% $$ \widehat S(\mu)B := \widehat S(\la)B + (\la-\mu)R(\mu,A)\widehat S(\la)B.$$
% Clearly this operator belongs to $\ga(H,E)$
% and by the resolvent identity we have
% \begin{align*}
% i_{-1} \widehat S(\mu)B 
% & = i_{-1}\widehat S(\la)B + (\la-\mu)i_{-1} R(\mu,A)\widehat S(\la)B \\
% & = i_{-1}\widehat S(\la)B + (\la-\mu) R(\mu,A_{-1})i_{-1}\widehat S(\la)B \\
% & = R(\lambda,A_{-1})\circ B + (\la-\mu) R(\mu,A_{-1}).R(\lambda,A_{-1}) B \\
% & = R(\mu,A_{-1}) B.  
% \end{align*}

Fix an arbitrary $\la>\omega_0(S_{-1})$, the exponential growth bound of
$(S_{-1}(t))_{t\ge 0}$. 
By assumption there exists an operator
$R_\infty \in \ga(L^2(\R_+;H),E)$ such that for all $x_{-1}^* \in
E_{-1}^*$ we have $R_\infty\s (i_{-1}\s x_{-1}\s) = B\s
S_{-1}\s(\cdot)x_{-1}\s$ in $ L^2(\R_+;H)$.  The operator $\widehat
S(\la)B: H\to E$ given by
\[ \widehat S(\la)B h := R_\infty(e^{-\la\cdot} \otimes h)\]
is $\ga$-radonifying and satisfies, for all $x_{-1}^*\in E_{-1}^*$,
\[ \lb i_{-1} \widehat S(\la)B h, x_{-1}^*\rb  = \int_0^\infty e^{-\la t} \lb S_{-1}(t)Bh, x_{-1}^*\rb \, dt 
= \lb R(\la, A_{-1})Bh,x_{-1}^*\rb.\]
Hence by the Hahn-Banach theorem, $\widehat S(\la)B$ satisfies the desired identity.
\end{proof}

If the semigroup generated by $A$ is analytic, then $R(\la,A_{-1})$ maps $E_{-1}$ into $\D(A_{-1}) = E$
and therefore we may interpret $R(\lambda,A_{-1}) B$ as an operator from $H$ to $E$.
By the injectivity of $i_{-1}$ this operator equals $\widehat S(\la)B$. 
From now on we simply write \[     R(\la,A)B   := \widehat S(\la)B \] to denote this operator.

\begin{proposition}\label{prop}
Suppose that $-A\in S(E)$ has a bounded $H^\infty$-calculus of angle $\om <\pihalbe$ on a Banach space 
$E$ with property $(\alpha)$.
Then for all  $B\in \calL(H, E_{-1})$ and $\theta \in (\omega, \pi)$
the following assertions are equivalent:
\begin{enumerate}
 \item \label{item:prop:AeinhalbB} $B\in \gamma(H,\dot E_{-\einhalb})$;
 \item \label{item:prop:allpsi} $t\mapsto \phi(-t A)B$ belongs to 
$\gamma(L^2(\R_+, \frac{dt}{t};H),\dot E_{-\einhalb})$
for all $\phi\in H_0^\infty(\Sigma_\theta)$;
 \item \label{item:prop:somepsi} $t\mapsto\psi(-t A)B$ belongs to 
$\gamma(L^2(\R_+, \frac{dt}{t};H),\dot E_{-\einhalb})$,
with $\psi(z) = z^{\einhalb}/(1+z)^{\dreihalb}$.
\end{enumerate}
In this situation, for any two $\phi,\tilde\phi\in H_0^\infty(\Sigma_\theta)$ satisfying
\[     \int_0^\infty \phi(t)\frac{dt}{t} =  \int_0^\infty \tilde\phi(t)\frac{dt}{t} = 1    \]
we have an equivalence of norms
\beq\label{eqNorms}
\n t\mapsto \phi(-tA)B\n_{\gamma(L^2(\R_+, \frac{dt}{t};H),\dot E_{-\einhalb})}\eqsim
\n t\mapsto \tilde\phi(-tA)B\n_{\gamma(L^2(\R_+, \frac{dt}{t};H),\dot E_{-\einhalb})}
\eeq
with implied constants independent of $\phi$ and $\tilde\phi$.
\end{proposition}

\begin{proof}
We shall prove the implications \ref{item:prop:AeinhalbB} $\Rightarrow$ 
\ref{item:prop:allpsi} $\Rightarrow$ \ref{item:prop:somepsi} 
$\Rightarrow$ \ref{item:prop:AeinhalbB}.

\medskip
\ref{item:prop:AeinhalbB} $\Rightarrow$ \ref{item:prop:allpsi}:
This follows from \cite[Theorem 7.2 and Remark 7.3(2)]{KaltonWeis:square-function-est} 
and \cite[Theorem 5.3]{NW-alpha}. 

\ref{item:prop:allpsi} $\Rightarrow$ \ref{item:prop:somepsi}: This is
trivial, as $\psi$ belongs to $H_0^\infty(\Sigma_\theta)$ for all $\theta<\pi$;

\ref{item:prop:somepsi} $\Rightarrow$ \ref{item:prop:AeinhalbB}:
 Let $(r_j)_{j\ge 1}$ be a Rademacher sequence on a probability space
$(\Omega,\P)$ and let $(h_j)_{j=1}^k$ be an orthonormal system in $H$.
Using that $\psi \in H_0^\infty(\Sigma_\theta)$, from \cite[Theorem~5.2.6]{Haase:Buch} we obtain
\[    
\begin{aligned} \sum_{j=1}^k r_j Bh_j & \eqsim \sum_{j=1}^k r_j 
\int_0^\infty (-tA)^\dreihalb (1-tA)^{-3} Bh_j\, \frac{dt}{t}
 \\ & =  \sum_{j=1}^k\sum_{n\in \Z} r_j \int_{2^n}^{2^{n+1}} (-tA)^\dreihalb (1-tA)^{-3} Bh_j\, \frac{dt}{t}
\end{aligned}
    \]
with convergence in 
$L^2(\Omega; E_{-1}) = L^2(\Omega;\dot E_{-1}) + L^2(\Omega; E)$ (cf. \eqref{eq:dot}).
Defining the vectors $x_n\in L^2(\Omega; E)\cap L^2(\Omega; \dot E_{-1})$ by
\[    x_n := \sum_{j=1}^k r_j \int_{2^n}^{2^{n+1}} (-tA)^\dreihalb (1-tA)^{-3} Bh_j\, \frac{dt}{t}    \]
and setting $m_N(t) = (2^{-n} t)^{\einhalb}$ for $t\in [2^n,2^{n+1})$, $n=-N,\dots,N$, and 
$m_N(t) = 0$ for $t\not\in [2^{-N}, 2^{N+1})$, 
we obtain (relative to the spaces $X_0 = L^2(\Omega;\dot E_{-1})$ and
$X_1 = L^2(\Omega;E)$) 
\begin{align*}
\  & \hskip-.2cm \mathscr{C}_0((x_n)_{n\in \Z})^2
\\ & = \sup_{N\ge 1}\tilde\EE \Big\n \sum_{j=1}^k\sum_{n=-N}^N r_j \tilde r_n 2^{-\nhalb} 
\int_{2^n}^{2^{n+1}} \!\!(-tA)^\dreihalb (1-tA)^{-3} Bh_j\, \frac{dt}{t}\Big\n_{L^2(\Omega;\dot{E}_{-1})}^2
\\ & = \sup_{N\ge 1}\tilde \EE \Big\n \sum_{j=1}^k\sum_{n=-N}^N  r_j \tilde r_n \int_{2^n}^{2^{n+1}} 
\!\!(2^{-n} t)^{\einhalb}
(-tA)(1-tA)^{-3} Bh_j\, \frac{dt}{t}\Big\n_{L^2(\Omega;\dot E_{-\einhalb})}^2
\\ & {=}\sup_{N\ge 1}
 \tilde\EE \Big\n \sum_{j=1}^k\sum_{n=-N}^N  r_j \tilde r_n \int_0^\infty \!\!
m_N(t)(-tA)(1-tA)^{-3} \one_{(2^n, 2^{n+1})}(t)B h_j\, \frac{dt}{t}\Big\n_{L^2(\Omega;\dot E_{-\einhalb})}^2
\\ & \eqsim  \sup_{N\ge 1}
 \EE' \Big\n \sum_{j=1}^k\sum_{n=-N}^N  r_{jn}' 
\int_0^\infty \!\! m_N(t)(-tA)(1-tA)^{-3} \one_{(2^n, 2^{n+1})}(t)B h_j\, \frac{dt}{t}\Big\n_{\dot E_{-\einhalb}}^2.
\end{align*}
In the last step, property $(\alpha)$ was used to pass from double Rademacher sums 
(on $(\Omega,\P)\times (\tilde\Omega, \tilde \P)$) to doubly indexed Rademacher sums 
(on some other probability space 
$(\Omega',\P')$).
Now, estimating Rademacher sums in terms of Gaussian sums we have
\begin{align*}
 \ & \mathscr{C}_0((x_n)_{n\in \Z})^2
 \\ & \qquad {\eqsim}\sup_{N\ge 1}
\EE' \Big\n \sum_{j=1}^k\sum_{n=-N}^N  \gamma_{jn}' 
\int_0^\infty m_N(t)(-tA)(1-tA)^{-3} \one_{(2^n, 2^{n+1})}(t)B h_j\, \frac{dt}{t}\Big\n_{\dot E_{-\einhalb}}^2\\
\intertext{Since the functions  $\one_{(2^n, 2^{n+1})}\otimes h_j$ in $L^2(\R_+, \frac{dt}{t};H)$ are orthonormal 
(up to the numerical constant $(\ln 2)^\einhalb$), 
one may estimate the above right-hand side by
}
 & \qquad {\lesssim} \sup_{N\ge 1}\n t\mapsto m_N(t)\phi(-t A) 
B\n_{\gamma(L^2(\R_+, \frac{dt}{t};H),\dot E_{-\einhalb})}^2
\end{align*}
where $\phi\in H_0^\infty(\Sigma_\theta)$ is given by $\phi(z) = z/(1+z)^3$.
Finally, using the Kalton--Weis $\gamma$-multiplier theorem and the
$\gamma$-boundedness of the operators
$(-tA)^{\einhalb}(1-tA)^{-\dreihalb}$, $t>0$, (which follows from Proposition \ref{prop:gamma-bounded})
we conclude that
\begin{equation}\label{eq:C0}
\begin{aligned}
 \mathscr{C}_0((x_n)_{n\in \Z})^2
 &  {\lesssim} \n t\mapsto \phi(-t A) B\n_{\gamma(L^2(\R_+, \frac{dt}{t};H),\dot E_{-\einhalb})}^2
\\ &  {\lesssim} \n t\mapsto \psi(-t A) B\n_{\gamma(L^2(\R_+, \frac{dt}{t};H),\dot E_{-\einhalb})}^2
\end{aligned}
\end{equation}
with $\psi(z) = z^{\einhalb}/(1+z)^{\dreihalb}$.

Similarly,
\[
\begin{aligned}
& \hskip-.2cm\mathscr{C}_1((x_n)_{n\in \Z})^2
\\ & = \sup_{N\ge 1}\tilde\EE\Big\n \sum_{j=1}^k\sum_{n=-N}^N  r_j \tilde r_n 2^{\nhalb} 
\int_{2^n}^{2^{n+1}} (-tA)^\dreihalb (1-tA)^{-3} Bh_j\, \frac{dt}{t}\Big\n_{L^2(\Omega;E)}^2
\\ & = \sup_{N\ge 1}\tilde\EE \Big\n \sum_{j=1}^k\sum_{n=-N}^N  r_j \tilde r_n
\\ & \qquad\qquad \times\int_0^\infty (2^{-n} t)^{-\einhalb}(-tA)^2(1-tA)^{-3} 
\one_{(2^n, 2^{n+1})}(t) Bh_j\, \frac{dt}{t}\Big\n_{L^2(\Omega;\dot E_{-\einhalb})}
^2
\\ & \lesssim_E  \n t\mapsto \tilde\phi(-t A) B\n_{\gamma(L^2(\R_+, \frac{dt}{t};H),\dot E_{-\einhalb})}^2
\\ & \lesssim_E \n t\mapsto \psi(-t A) B\n_{\gamma(L^2(\R_+, \frac{dt}{t};H),\dot E_{-\einhalb})}^2
\end{aligned}
\]
with $\tilde \phi(z) = z^2/(1+z)^3$ and $\psi(z) =z^{\einhalb}/(1+z)^{\dreihalb}$ as before.

By Proposition~\ref{prop:RadInterp} and estimating Gaussian sums by Rademacher sums, this proves that
\begin{align*}
 \Big\n \sum_{j=1}^k \ga_j Bh_j\Big\n_{L^2(\Omega;\dot E_{-\einhalb})} 
&\eqsim_E \Big\n \sum_{j=1}^k r_j Bh_j\Big\n_{L^2(\Omega;\dot E_{-\einhalb})} 
\\ & \lesssim_E \n t\mapsto \psi(-t A) B\n_{\gamma(L^2(\R_+, \frac{dt}{t};H),\dot E_{-\einhalb})}.
\end{align*}
Taking the supremum over all finite orthonormal systems in $H$ and using that
$E$ has property $(\alpha)$ and therefore does not contain an isomorphic copy of $c_0$,
we obtain (using a theorem of Hoffmann-J\o rgensen and Kwapie\'n, see \cite[Theorem 4.3]{NeeCanberra}) 
that $B$ is $\gamma$-radonifying as an operator from $H$ into $\dot E_{-\einhalb}$ and
\[     \n B\n_{\gamma(H,\dot E_{-\einhalb})} \lesssim 
\n t\mapsto\psi(-t A) B\n_{\gamma(L^2(\R_+, \frac{dt}{t};H),\dot E_{-\einhalb})}.    \]

We have now proved the equivalences \ref{item:prop:AeinhalbB} $\Leftrightarrow$ \ref{item:prop:allpsi} 
$\Leftrightarrow$
\ref{item:prop:somepsi}.
 It remains to check that these equivalent 
conditions imply the norm equivalence \eqref{eqNorms}.
Let $\mu$ be the centred Gaussian
measure on $\dot E_{-\einhalb}$ associated with the $\ga$-radonifying operator $B\in\ga(H,\dot E_{-\einhalb})$.
Suppose $\phi,\tilde\phi\in H_0^\infty(\Sigma_\theta)$ are nonzero functions. By
\cite[Theorems 5.2, 5.3]{NW-alpha}, assertion \ref{item:main:existence-inv-measure} implies 
\begin{align*}
\n t\mapsto \phi(-tA)B\n_{\gamma(L^2(\R_+, \frac{dt}{t};H),\dot E_{-\einhalb})}
 & \eqsim\int_{\dot E_{\einhalb}}\n t\mapsto 
\phi(-tA)x\n_{\gamma(L^2(\R_+, \frac{dt}{t}),\dot E_{-\einhalb})}\,d\mu(x)
\\ & \stackrel{(1)}{\eqsim}\int_{\dot E_{\einhalb}}\n t\mapsto 
\tilde\phi(-tA)x\n_{\gamma(L^2(\R_+, \frac{dt}{t}),\dot E_{-\einhalb})}\,d\mu(x)
\\ & \eqsim \n t\mapsto \tilde\phi(-tA)B\n_{\gamma(L^2(\R_+, \frac{dt}{t};H),\dot E_{-\einhalb})}.
\end{align*}
Here, step $(1)$ follows from 
\cite[Proposition 7.7]{KaltonWeis:square-function-est}. The implied constants are independent 
of $\phi$ and $\tilde\phi$ under the normalisation as stated in the proposition. 
\end{proof}

\begin{remark}
The only step in the proof where we made use of the boundedness of the
functional calculus is the Rademacher interpolation argument. For all
other parts, $\ga$-sectoriality of angle less than $\pihalbe$ is
sufficient. However, one actually needs only the continuous embedding
$$ \langle L^2(\Om; E), L^2(\Om; \dot
E_{-1})\rangle_\einhalb  \hookrightarrow L^2(\Om;\dot E_{-\einhalb}) $$
instead of an equality. As in Proposition~\ref{prop:RadInterp}
this boils down to having the embedding for the underlying Banach spaces
$  \langle E, \dot E_{-1}\rangle_\einhalb \hookrightarrow \dot
E_{-\einhalb}$. An
inspection of the proof of
\cite[Theorems~4.1~and~7.4]{KaltonKunstmannWeis} shows that the latter
embedding does not require the full power of the boundedness of the
functional calculus but merely a (discrete dyadic) square function
estimate of the form
\[
\sup_{\epsilon_k=\pm 1} \Bigl\n \sum_k \epsilon_k \varphi(2^k
A^{\sharp})x \Bigr\n
\lesssim \n x\n
\]
for
some $\varphi \in H_0^\infty(\Sigma_\theta)$ for $\theta \in (0,
\pi)$, where $A^{\sharp}$ denotes the part of $A^\ast$ in $E^\sharp =
\overline{ \D( A^*) } \cap \overline{ \Ran(A^*) }$ (the closures
are taken in the strong
topology of $E^*$). 
These `dual' square function estimates match the hypothesis in
Le~Merdy's theorem on the Weiss conjecture
\cite[Theorem~4.1]{LeMerdy:weiss-conj} in the sense that Le~Merdy
treats observation operators and requires upper square function
estimates for $A$ whereas we treat control operators and therefore
need `dual' square function estimates. The construction of $A^\sharp$
instead of $A^*$ is needed when non-reflexive Banach spaces are
concerned. On reflexive spaces one has $A^\sharp{=} A^*$, and the explained
duality with Le~Merdy's result is more apparent.
\end{remark}

In the next lemma, $\widehat f$ denotes the Laplace transform of a function $f$.

\begin{lemma}[Laplace transforms]\label{lem:Laplace}
For all $f\in L^2(\R_+, \frac{dt}{t};H)$,
the function $Lf(t):= t \widehat f(t)$
belongs to $L^2(\R_+,\frac{dt}{t};H)$ and
\[    \n Lf\n_{L^2(\R_+,\frac{dt}{t};H)}
\le \n f \n_{L^2(\R_+,\frac{dt}{t};H)}.    \]
\end{lemma}
\begin{proof}
 By the Cauchy-Schwarz inequality,
\begin{align*}
\int_0^\infty  t^2 \n\widehat f(t)\n_H ^2\,\frac{dt}{t}
& = \int_0^\infty  \Big\n\int_0^\infty f(s)\, te^{-st}\,ds\Big\n_H^2\,\frac{dt}{t} \\
& \le \int_0^\infty  \int_0^\infty \n f(s)\n_H^2\, te^{-st}\,ds\,\frac{dt}{t} \\
& = \int_0^\infty  \int_0^\infty \n f(s)\n_H^2\, e^{-st}\,dt\, ds  =
\int_0^\infty \n f(s)\n_H^2\,\frac{ds}{s}. \qedhere
\end{align*}
\end{proof}

As a consequence, the mapping $L: f\mapsto Lf$ 
is a contraction on $L^2(\R_+,\frac{dt}{t};H)$.  By the Kalton--Weis
extension theorem, $L$ extends to a linear contraction on the space $
\gamma(L^2(\R_+,\frac{dt}{t};H),E)$, for any Banach space $E$.

\begin{proof}[Proof of the equivalences {\rm
    \ref{item:main:existence-inv-measure} $\Leftrightarrow$
    \ref{item:main:A-onehalf-B-in-gamma} $\Leftrightarrow$
    \ref{item:main:weiss-condition}} of
  Theorem~\ref{thm:equivalence}.]

\medskip

\ref{item:main:existence-inv-measure} $\Rightarrow$
\ref{item:main:A-onehalf-B-in-gamma}: By assumption, $t\mapsto S(t)B$
belongs to $\gamma(L^2(\R_+;H),E)$. It follows that $t\mapsto
\eta(-tA)B$ belongs to $ \gamma(L^2(\R_+,\frac{dt}{t};H),\dot
E_{-\einhalb})$, with $\eta(z) = z^\einhalb \exp(-z)$.  The Laplace
transform of $t\mapsto (tz)^\einhalb \exp(-tz)$ equals $\la\mapsto
\einhalb \sqrt{\pi} z^\einhalb (\lambda+z)^{-\dreihalb}$.  Hence, by
\cite[Lemma 9.12]{KunstmannWeis:Levico} or by using the Phillips
calculus (see \cite{Haase:Buch}),
$$\einhalb\sqrt{\pi}(-A)^\einhalb(\la-A)^{-\dreihalb}B = \int_0^\infty e^{-\la t} (-tA)^{\einhalb} S(t)B\,dt,$$
or, equivalently,
$$\einhalb\sqrt{\pi} (-A/\la)^\einhalb(1-A/\la)^{-\dreihalb}B = \la\int_0^\infty e^{-\la t} \eta(-tA)B\,dt.$$
By Lemma~\ref{lem:Laplace} and the remark following it, we obtain that
$ \la \mapsto (-A/\la)^\einhalb(1-A/\la)^{-\dreihalb}B$ belongs to 
$ \gamma(L^2(\R_+,\frac{d\la}{\la};H),\dot E_{-\einhalb})$.
Upon substituting $1/\la =\mu$ we find that
$\mu\mapsto \psi(-\mu A)B$ belongs to $ \gamma(L^2(\R_+,\frac{d\mu}{\mu};H),\dot E_{-\einhalb})$ with
$\psi(z) = z^{\einhalb}/(1+z)^{\dreihalb}$.
Now \ref{item:main:A-onehalf-B-in-gamma} follows as an application of Proposition~\ref{prop}.

\ref{item:main:A-onehalf-B-in-gamma} $\Rightarrow$ \ref{item:main:weiss-condition}:
From Proposition~\ref{prop} we get that
$t\mapsto (-tA)^{\einhalb} (1-tA)^{-1}B$ belongs to
$\gamma(L^2(\R_+,\frac{dt}{t};H), \dot E_{-\einhalb})$, or equivalently, that
$t\mapsto t^{\einhalb} (1-tA)^{-1}B$ belongs to
$\gamma(L^2(\R_+,\frac{dt}{t};H), E)$. Substituting $t = 1/s$ 
we obtain that $s\mapsto s^{\einhalb}(s-A)^{-1}B $ belongs to
$\gamma(L^2(\R_+,\frac{dt}{t}; H), E)$.

\ref{item:main:weiss-condition} $\Rightarrow$ \ref{item:main:A-onehalf-B-in-gamma}: By
substituting $t = 1/s$ the assumption implies that $s\mapsto s^{\einhalb}(1-sA)^{-1}B $ belongs to
$\gamma(L^2(\R_+,\frac{dt}{t}; H), E)$, or equivalently, that $s\mapsto (-sA)^{\einhalb}(1-sA)^{-1}B $ belongs to
$\gamma(L^2(\R_+,\frac{dt}{t};H), \dot E_{-\einhalb})$. Then by the $\gamma$-multiplier lemma 
(using that the operators
$(1-sA)^{-{\einhalb}}$, $s>0$, are $\gamma$-bounded by Proposition \ref{prop:gamma-bounded}),
we obtain that assumption \ref{item:prop:somepsi} of Proposition~\ref{prop} is satisfied.

\smallskip
\ref{item:main:A-onehalf-B-in-gamma} $\Rightarrow$ \ref{item:main:existence-inv-measure}: By Proposition~\ref{prop},
$t\mapsto (-tA)^{\einhalb}\exp(tA) B = (-tA)^{\einhalb}S(t)B$ belongs to
$\ga(L^2(\R_+,\frac{dt}{t};H),\dot E_{-\einhalb})$. This is equivalent to
saying that $t\mapsto S(t)B$ belongs to $\ga(L^2(\R_+;H),E)$.
\end{proof}

For the proofs of the implications \ref{item:main:A-onehalf-B-in-gamma} $\Rightarrow $ 
\ref{item:main:dyadic-condition} $\Rightarrow$ \ref{item:main:weiss-condition}   
we need some further preparations. 

An interval in $\R_+$ will be called {\em dyadic} (with respect to the measure $\frac{dt}{t}$) 
if it is of the form $[2^{k/2^M}, 2^{(k+1)/2^M})$ with $M\in\N$ and $k\in\Z$.
 
\begin{lemma}\label{lem:dyadische-mittel}
Let $-A\in S(E)$ be $\ga$-sectorial 
and let $I_1,\dots,I_N$ be 
dyadic intervals. 
For any choice of the numbers $s_n, t_n\in I_n$ 
we have the equivalence
\begin{align*}
 \Big\n \sum_{n\in F} \gamma_n   s_n^{\einhalb} R(s_n,A) B\Big\n_{L^2(\Omega;\ga(H,E))}
\eqsim
 \Big\n \sum_{n\in F} \gamma_n  t_n^{\einhalb} R(t_n,A) B\Big\n_{L^2(\Omega;\ga(H,E))}
\end{align*}
with constants independent of the finite subset $F\subseteq \Z$, the intervals $I_n$, and the choice of $s_n,t_n$.
\end{lemma}
\begin{proof}
First note that, since $I_n$ is dyadic,
$|s_n^\einhalb\pm t_n^\einhalb|\le 4\max\{s_n^\einhalb, t_n^\einhalb\}$.

We have, using the resolvent identity, the $\ga$-boundedness of the operators $t R(t,A)$ for $t>0$, 
and the contraction principle,
\begin{align*}
&\Big\n\sum_{n\in F} \gamma_n(s_n^{\einhalb} R(s_n,A)-t_n^{\einhalb} R(t_n,A))B\Big\n_{L^2(\Omega;\ga(H,E))}\\
&\quad \le\Big\n\sum_{n\in F} 
\gamma_n\frac{t_n-s_n}{t_n^{\einhalb}s_n^{\einhalb}}s_nR(s_n,A)t_n^{\einhalb}R(t_n,A)B\Big\n_{L^2(\Omega;\ga(H,E))}\\
&\quad \quad +\Big\n\sum_{n\in F} 
\gamma_n\frac{s_n^{\einhalb}-t_n^{\einhalb}}{t_n^{\einhalb}}t_n^{\einhalb}R(t_n,A)B\Big\n_{L^2(\Omega;\ga(H,E))}\\
&\quad \lesssim \Big\n\sum_{n\in F} \gamma_nt_n^{\einhalb} R(t_n,A)B\Big\n_{L^2(\Omega;\ga(H,E))}.
\end{align*}
By the triangle inequality in $L^2(\Omega;\ga(H,E))$ it then follows that
$$\Big\n\sum_{n\in F} \gamma_ns_n^{\einhalb} R(s_n,A)B\Big\n_{L^2(\Omega;\ga(H,E))}
\lesssim \Big\n\sum_{n\in F} \gamma_nt_n^{\einhalb} R(t_n,A)B\Big\n_{L^2(\Omega;\ga(H,E))}.$$
The converse inequality is obtained by reversing the roles of $s_n$ and $t_n$. 
\end{proof}

\begin{lemma}\label{lem:Poisson}
Let $f:\Sigma_\theta\to H$ be a bounded analytic function and suppose that, for some $0<\eta<\theta$,
the functions $t\mapsto f(e^{\pm i\eta}t)$ belong to $L^2(\R_+, \frac{dt}{t};H)$.
Then $$\sum_{n\in\Z} \n f(2^n)\n_H^2 < \infty.$$ 
\end{lemma}
\begin{proof}
  Since $f$ is continuous we may suppose that $H$ is separable. By
  expanding the values of $f$ with respect to an orthonormal basis
  in $H$, it suffices to prove the lemma for
  the case $H$ equals the scalar field. 

  By considering $g(z) = f(\exp(z))$, we may reformulate the problem on
  the strip $S_\theta = \{z\in\C: \ |\Im z| < \theta\}$.  The
  objective is then to show that if the restriction of a bounded
  analytic function $g$ on $S_\theta$ to the lines $\Im z = \pm \eta$
  belongs to $L^2(\R)$, then $\sum_{n\in\Z} |g(n \ln 2)|^2 < \infty.$ 
  The proof of this uses the following standard technique.  By the Poisson
  formula for the strip we have 
\[
\sup_{|\zeta|<\eta} \big\n g|_{\{\Im z = \zeta\}}\big\n_2 < \infty
\]
and therefore  $g|_{S_\eta} \in L^2(S_\eta)$.  For $0<\delta<\eta$  
consider the discs  
\[
Q_n = \{z\in  \C: |z - n\ln 2|<\delta\}, \qquad n\in\Z,
\]
centred around $n \in \ZZ$. Taking $\delta$ small enough, the
functions $\phi_n = |Q_n|^{-\einhalb} \one_{Q_n}$ have disjoint
support and are hence orthonormal in $L^2(S_\eta)$. By the mean value
theorem we obtain 
\begin{align*} \sum_{n\in\Z} |g(n\ln 2)|^2
& = \sum_{n\in\Z} \Big|\frac1{|Q_n|}\int_{Q_n} g(x+iy)\,dx\,dy\Big|^2
\\ & =  \frac1{\pi \delta^2} 
       \sum_{n\in\Z} \Big|\int_{S_\eta} g(x+iy) \phi_n(x+iy) \,dx\,dy  \Big|^2
\\ & \le \frac1{\pi \delta^2} 
        \bigl\n g|_{S_\eta}\bigr\n_{L^2(S_\eta)}^2. \qedhere
\end{align*}
\end{proof}

This lemma can be restated as saying that the mapping $f\mapsto (f(2^n))_{n\in\Z}$
is bounded from the weighted  Hardy space $H^2(\Sigma_\eta, \mu;H)$ to $\ell^2(H)$,
where $\mu$ is the image on the sector 
$\Sigma_\eta$ of the Lebesgue measure on the strip 
$S_\eta$ under the exponential mapping;
note that Lebesgue measure on horizontal lines in the strip $S_\eta$
is mapped to the measure $dt/t$ on rays emanating from the origin in the sector $\Sigma_\eta$. 

By the Kalton--Weis extension theorem, this mapping extends to a
bounded operator from $\ga(H^2(\Sigma_\eta, \mu;H),E)$ to
$\ga(\ell^2(H),E)$, for any Banach space $E$.  This is what will be
needed below.

\begin{proof}[End of the proof of Theorem~\ref{thm:equivalence}]
We shall now prove the remaining implications \ref{item:main:A-onehalf-B-in-gamma} $\Rightarrow $ 
\ref{item:main:dyadic-condition} $\Rightarrow$ \ref{item:main:weiss-condition}.
   
We begin with the proof of \ref{item:main:A-onehalf-B-in-gamma} $\Rightarrow$ \ref{item:main:dyadic-condition}.
First of all, Lemma \ref{lem:resolvB} implies that $R(t,A)B \in \gamma(H,E)$ for all $t>0$. 
By the implication \ref{item:main:A-onehalf-B-in-gamma} $\Rightarrow$ \ref{item:main:weiss-condition} applied to the
operators $e^{\pm i\theta}A$ for a sufficiently small $\theta>0$ we find that the 
functions $$t \mapsto t^\einhalb R(t,e^{\pm i\theta}A)B = 
 e^{\mp i\theta}\,t^\einhalb R(te^{\mp i\theta},A)B$$
belong to $\ga( L^2(\R_+,\frac{dt}{t};H), E)$. 
By Lemma \ref{lem:Poisson} and the remark following it, 
we obtain that the sequence $(2^{\nhalb}R(2^n,A)B)_{n\in \Z}$ belongs to $\ga(\ell^2(H),E)$.
But this is the same as saying that \ref{item:main:dyadic-condition} holds.

\smallskip
We turn to the proof of
\ref{item:main:dyadic-condition} $\Rightarrow$
\ref{item:main:weiss-condition}.
Let $S_{nm}^{(M)}$ denote the average of $t^{\einhalb}R(t,A)$ (with
respect to $dt/t)$
over the dyadic interval $I_{nm}^{(M)} = [ 2^{n+m 2^{-M}},
2^{n+(m+1)2^{-M}})$.
Let $t_{nm}^{(M)} = 2^{n+m 2^{-M}}$ be the left endpoint of the interval
$I_{nm}^{(M)}$. Then
\begin{align*}
S_{nm}^{(M)}
& = \fint_{I_{nm}^{(M)}} t^{\einhalb} R(t,A) B  \frac{dt}{t}
\\ & = \fint_{I_{nm}^{(M)}} t^{\einhalb} \big( R(t, A)(t_{nm}^{(M)} - A)
\big) R(t_{nm}^{(M)},A) B  \frac{dt}{t}\\
& = \Big( \fint_{I_{nm}^{(M)}} \frac{ t^\einhalb}{ (t_{nm}^{(M)})^\einhalb} \big(
   \frac{t_{nm}^{(M)}}{t} \cdot tR(t, A) - A R(t, A) \big)
 \frac{dt}{t}\Big) \circ [(t_{nm}^{(M)})^{\einhalb}R(t_{nm}^{(M)},A) B]\\
& =: U_{nm}^{(M)} \circ [(t_{nm}^{(M)})^{\einhalb}R(t_{nm}^{(M)},A) B].
\end{align*}
Since $t/ t_{nm}^{(M)} \in [1, 2]$ on $I_{nm}^{(M)}$, the operators
$U_{nm}^{(M)}$ belong (up to a constant) to the closure of the absolute
convex hull of $\{ A R(t,A),\, tR(t,A)
\suchthat t>0\}$. By $\ga$-sectoriality of $A$ (which follows from Proposition 
\ref{prop:gamma-bounded}) this
family is $\ga$-bounded.

Fix a finite set $F\subseteq \Z$. Then,
\begin{align*}
\ & \Big\n  \sum_{n\in F}
\sum_{m=0}^{2^{M}-1}   \one_{I_{nm}^{(M)} }\otimes
S_{nm}^{(M)}B\Big\n_{\ga(L^2(\R_+,\frac{dt}{t};H),E)}
\\ & \qquad \stackrel{(1)}{\eqsim} \Big\n  \sum_{n\in F}
\sum_{m=0}^{2^{M}-1}   \one_{I_{nm}^{(M)} }\otimes S_{nm}^{(M)}B
\Big\n_{\ga(L^2(\R_+,\frac{dt}{t}),\ga(H,E))}
\\ & \qquad \stackrel{(2)}{\eqsim} \frac1{2^{\Mhalb}} \Big\n \sum_{n\in
F} \sum_{m=0}^{2^{M}-1}  \gamma_{nm}
S_{nm}^{(M)}B\Big\n_{L^2(\Omega;\ga(H,E))}
\\ & \qquad \stackrel{(3)}{\lesssim} \frac1{2^{\Mhalb}}  \Big\n \sum_{n\in
F} \sum_{m=0}^{2^{M}-1} \gamma_{nm}
(t_{nm}^{(M)})^\einhalb R(t_{nm}^{(M)},A) B\Big\n_{L^2(\Omega;\ga(H,E))}
\\ & \qquad \stackrel{(4)}{\eqsim} \frac1{2^{\Mhalb}}  \Big\n \sum_{n\in
F} \sum_{m=0}^{2^{M}-1} \gamma_{nm}  2^{\nhalb} R(2^n,A)
B\Big\n_{L^2(\Omega;\ga(H,E))}
\\ & \qquad \stackrel{(5)}{=} \Big\n \sum_{n\in F}\sum_{m=0}^{2^{M}-1}
\one_{I_{nm}^{(M)} }\otimes 2^{\nhalb} R(2^n,A)
B\Big\n_{\ga(L^2(\R_+,\frac{dt}{t}),\ga(H,E))}
\\ & \qquad = \Big\n \sum_{n\in F} \one_{I_{n} }\otimes 2^{\nhalb} R(2^n,A)
B\Big\n_{\ga(L^2(\R_+,\frac{dt}{t}),\ga(H,E))}
\\ & \qquad \stackrel{(6)}{\eqsim} \Big\n \sum_{n\in F} \gamma_n
2^{\nhalb} R(2^n,A) B\Big\n_{L^2(\Omega;\ga(H,E))}
\end{align*}
with implicit constants independent of $F$ and $M$. In this
computation, (1) follows from property $(\alpha)$; (2), (5), (6) from
the identity \eqref{eq:finite-rank} along with the fact that the
dyadic interval $I_{nm}^{(M)}$ has
$dt/t$-measure $\eqsim 2^{-M}$;
Estimate (3) follows from the $\ga$-boundedness of the operators
$U_{nm}^{(M)}$;
and (4) from
Lemma \ref{lem:dyadische-mittel} applied to the points $s_n{=}2^n$ and
$t_{nm}^{(M)}$ in $I_n = [2^n, 2^{n+1})$.

By the $\gamma$-Fatou lemma (see \eqref{eq:gFatou}),
the above estimate
implies \ref{item:main:weiss-condition}. 
\end{proof}

\let\H\trueH
\bibliographystyle{plain}
\bibliography{stochweiss}

\end{document}